\documentclass[11pt]{article}
\usepackage[margin=1in]{geometry}
\usepackage{graphicx,xcolor,cite,mathtools}
\usepackage{amssymb,amsmath,amsthm,mathrsfs,paralist,bm,esint,setspace}
\RequirePackage[colorlinks,citecolor=blue,urlcolor=blue]{hyperref}

\newcommand{\R}{{\mathbb{R}}}
\newcommand{\E}{\mathrm{E}}

\renewcommand{\P}{\mathrm{P}}
\renewcommand{\d}{\mathrm{d}}

\newcommand{\e}{\mathrm{e}}

\DeclareMathOperator{\Cov}{\text{\rm Cov}}

\DeclarePairedDelimiter\floor{\lfloor}{\rfloor}

\title{Lower bound on spatial asymptotic of parabolic Anderson model with narrow wedge initial condition}
\author{ Fei  Pu
	}
\date{}                                           

\begin{document}
\newtheorem{stat}{Statement}[section]
\newtheorem{proposition}[stat]{Proposition}
\newtheorem*{prop}{Proposition}
\newtheorem{corollary}[stat]{Corollary}
\newtheorem{theorem}[stat]{Theorem}
\newtheorem{lemma}[stat]{Lemma}
\theoremstyle{definition}
\newtheorem{definition}[stat]{Definition}
\newtheorem*{cremark}{Remark}
\newtheorem{remark}[stat]{Remark}
\newtheorem*{OP}{Open Problem}
\newtheorem{example}[stat]{Example}
\newtheorem{nota}[stat]{Notation}
\numberwithin{equation}{section}
\maketitle

\begin{abstract}
Let $\{u(t\,,x): (t,x)\in (0, \infty)\times \R\}$ be the solution to parabolic Anderson model with narrow wedge initial condition. Using the association property of parabolic Anderson model, we establish a lower bound on  spatial asymptotic of the solution: 
          \begin{align*}
          \liminf_{R\to\infty}\frac{ \max_{|x|\leq R}\left(\log u(t\,,x) + \frac{x^2}{2t}\right)}{(\log R)^{2/3}} \geq \frac14\left(\frac{t}2\right)^{1/3}, \quad \text{a.s.}
          \end{align*}

          \end{abstract}

\bigskip\bigskip



{
}

\section{Introduction}\label{sec:int}

Consider the parabolic Anderson model with narrow wedge initial condition
\begin{align}\label{PAM}
\begin{cases}
\partial_t u(t\,,x) = \frac12\partial_x^2 u(t\,,x) + u(t\,,x)\,\xi(t\,,x), \\
u(0)=\delta_0,
\end{cases}
\end{align}
where $\xi$ denotes the space-time white noise. Conus et al \cite{CJK13}  studied the growth of the tallest peaks of solution to stochastic heat equation. For the parabolic Anderson model with flat initial data, i.e., $u(0)\equiv1$, they showed that there exist positive constants $c_1$ and $c_2$ such that almost surely
\begin{align}\label{eq:asym}
c_1\leq \liminf_{R\to\infty}\frac{\max_{|x|\leq R}\log u(t\,,x)}{(\log R)^{2/3}}\leq \limsup_{R\to\infty}\frac{\max_{|x|\leq R}\log u(t\,,x)}{(\log R)^{2/3}} \leq c_2, 
\end{align}
(see \cite[Theorem 1.3]{CJK13}). Later on, Chen \cite{Che16} obtained the precise asymptotic form, that is, 
\begin{align*}
\lim_{R\to\infty}\frac{\max_{|x|\leq R}\log u(t\,,x)}{(\log R)^{2/3}}=\frac34\left(\frac{2t}{3}\right)^{1/3}, \quad \text{a.s.}
\end{align*}
(see \cite[Corollary 1.5]{Che16}). 

Huang and L\^e \cite{HuL19} have investigated the spatial asymptotic of the stochastic heat equation with compactly supported initial data. In the case of narrow wedge initial data, i.e., $u(0)=\delta_0$, we see from \cite[Theorem 1.4]{HuL19} that
\begin{align*}
\limsup_{R\to\infty}\frac{ \max_{|x|\leq R}\left(\log u(t\,,x) + \frac{x^2}{2t}\right)}{(\log R)^{2/3}}\leq \frac34\left(\frac{2t}{3}\right)^{1/3}, \quad \text{a.s.}
\end{align*}
A lower bound in \cite[Theorem 1.4]{HuL19} is missing and it is conjectured in Huang and L\^e \cite{HuL19} (see \cite[Conjecture 1.2 and (1.18)]{HuL19}) that, in the case of narrow wedge initial data,
\begin{align*}
\lim_{R\to\infty}\frac{ \max_{|x|\leq R}\left(\log u(t\,,x) + \frac{x^2}{2t}\right)}{(\log R)^{2/3}}= \frac34\left(\frac{2t}{3}\right)^{1/3}, \quad \text{a.s.}
\end{align*}

For the flat initial data,  Conus et al \cite{CJK13} established the lower bound in \eqref{eq:asym} using a localization argument. In the situation of the narrow wedge initial data, this localization procedure does not apply because of some technical difficulty. We refer to \cite[Section 6.2]{HuL19} for detailed explanation on this.  As an attempt to obtain the exact spatial asymptotics for narrow wedge initial data, Huang and L\^e \cite{HuL19} consider a sequence of approximating solution and use the Gaussian nature of the noise to give some estimate on the spatial asymptotic of this approximating solution; see \cite[Theorem 1.5]{HuL19}.

The goal of this paper is to establish a lower bound on the spatial asymptotic of the parabolic Anderson model with narrow wedge initial data. We have the following. 
\begin{theorem}\label{th:main}
          Let $\{u(t\,,x): (t,x)\in \R_+\times \R\}$ be the solution to parabolic Anderson model \eqref{PAM}. For fixed $t>0$, we have almost surely, 
          \begin{align}\label{eq:lower}
          \liminf_{R\to\infty}\frac{ \max_{|x|\leq R}\left(\log u(t\,,x) + \frac{x^2}{2t}\right)}{(\log R)^{2/3}} \geq \frac14\left(\frac{t}2\right)^{1/3}.
          \end{align}
\end{theorem}

Conus et al \cite{CJK13} introduced a coupled process to approximate the solution to stochastic heat equation. This localization argument implies that the random variables $u(t\,,x)$ and $u(t\,,y)$ are approximately independent when the
spatial variables $x$ and $y$ are far away, where $u$ denotes the solution to parabolic Anderson model with flat initial condition. Combined with the asymptotic behavior of the tail probability of the solution and Borel-Cantelli's lemma, it leads 
to the lower bound in \eqref{eq:asym}. Our approach to the independence structure of the solution to parabolic Anderson model \eqref{PAM} relies on the association property of the solution, which has been proved in \cite{CKNP23}.  The association property and boundedness of the density of parabolic Anderson model (see Proposition \ref{prop:density} and
Corollary \ref{cor:density2} below) together lead to  a crucial probability inequality \eqref{eq:prob}, which implies that the independence structure of the solution can be characterized by the asymptotic behavior of the covariance of the solution, and hence enables us to study the lower bound on spatial asymptotic of parabolic Anderson model with narrow wedge initial condition. 

We will introduce some basic facts on Malliavin calculus and association in Section \ref{sec:pre}. The boundedness of density is proved in Section \ref{sec:bound}. Finally, we prove Theorem \ref{th:main} in Section \ref{sec:main}.

Let us close the Introduction with a brief description of the notation of this paper.
For every $Z\in L^k(\Omega)$, we write $\|Z\|_k$ instead of
$(\E[|Z|^k])^{1/k}$.  Throughout we write ``$g_1(x)\lesssim g_2(x)$ for all $x\in X$'' when
there exists a real number $L$ such that $g_1(x)\le Lg_2(x)$ for all $x\in X$.
Alternatively, we might write ``$g_2(x)\gtrsim g_1(x)$ for all $x\in X$.'' By
``$g_1(x)\asymp g_2(x)$ for all $x\in X$'' we mean that $g_1(x)\lesssim g_2(x)$
for all $x\in X$ and $g_2(x)\lesssim g_1(x)$ for all $x\in X$.

\section{Preliminaries} \label{sec:pre}

Following Walsh \cite{Wal86}, we interpret the stochastic PDE  in
the following mild form:
\begin{align}\label{mild3}
	u(t\, , x) = \bm{p}_{t}(x) + \int_0^t\int_\R\bm{p}_{t - s}(x - y)u(s\,, y)\,\xi(\d s\,\d y),
\end{align}
where
\[
	\bm{p}_{t}(x) = \frac{1}{\sqrt{2\pi t}}\,\e^{-x^2/(2t)}
	\qquad\text{for all }t > 0\text{ and }x \in \R.
\]
Let
\begin{equation} \label{U}
	U(t\,,x):= \frac{u(t\,,x)}{\bm{p}_t(x)}
	\qquad\text{for all $t>0$ and $x\in\R$}.
\end{equation}
By \cite[Lemma A.4]{CKNP22}, $\lim_{t\downarrow0}U(t\,,x)=1$ in $L^k(\Omega)$
for all $x\in\R$ and $k\ge 2$.
Therefore, we also define
\[
	U(0\,,x) :=1\qquad\text{for all $x\in\R$},
\]
throughout.
Amir et al \cite[Proposition 1.4]{ACQ11} showed
that the process $U(t):=\{U(t\,,x)\}_{x\in\R}$ is stationary for every $t>0$.
The formulation \eqref{mild3} of the stochastic PDE \eqref{PAM}
can be recast equivalently in terms of $U$ as follows:
\[
	U(t\,,x) = 1 + \int_0^t\int_\R\frac{\bm{p}_{t - s}(x - y) \bm{p}_s(y)}{\bm{p}_t(x)}\,
	U(s\,,y)\,\xi(\d s\,\d y).
\]
Because
\begin{equation}\label{PPPP}
	\frac{\bm{p}_{t-s}(a)\bm{p}_s(b)}{\bm{p}_t(a+b)} =
	\bm{p}_{s(t-s)/t}\left( b - \frac st (a+b)\right)
	\quad\text{for all $0<s<t$ and $a,b\in\R$}, 
	\end{equation}
equation \eqref{mild3} can be recast as the following random evolution equation for $U$:
\begin{align}\label{E:U}
	U(t\,,x) = 1 +\int_0^t\int_\R
	 \bm{p}_{s(t-s)/t}\left(y- \frac{s}{t}x \right)U(s\,,y)\,
	\xi(\d s\, \d y).
\end{align}

 Chen and Dalang \cite[Theorem 2.4]{ChD15} found
	that there exists a real number $c_{T, k} > 0$ such that for all $(t\,, x)
	\in (0\,, T] \times \R$ and $k\geq 2$,
	\begin{equation*}
            \|u(t\,, x)\|_k \leq c_{T, k}\, \bm{p}_t(x),
	\end{equation*}
	which implies that 
	\begin{align}\label{moment}
	\sup_{(t,x)\in [0, T]\times \R}\|U(t\,,x)\|_k \leq c_{T, k}.
	\end{align}

\subsection{Malliavin calculus}

Let $\mathcal{H}= L^2(\R_+ \times \R)$.
The Gaussian family $\{ W(h)\}_{h \in \mathcal{H}}$ formed by the Wiener integrals
\[
	W(h)= \int_0^\infty\int_\R  h(s\,,y)\, \xi(\d s\, \d y)
\]
defines an  {\it isonormal Gaussian process} on the Hilbert space $\mathcal{H}$.
In this framework we can develop the Malliavin calculus (see Nualart \cite{Nua06}).
We denote by $D$ the derivative operator.  Using Clark-Ocone formula  (see Chen et al \cite[Proposition 6.3]{CKNP21}), we can derive the following Poincar\'e-type inequality
\begin{equation} \label{Poincare:Cov}
	|\Cov(F\,, G)| \le \int_0^{\infty}\d s \int_\R\d y\
	\left\| D_{s,y } F  \right\|_2
	\left\| D_{s,y}G \right\|_2
	\qquad\text{for all $F,G\in\mathbb{D}^{1,2}$.}
\end{equation}
(see \cite[(2.1]{CKNP22}).

The divergence operator $\delta$ is defined as the adjoint of the derivative operator $D$ as an unbounded operator from $L^1(\Omega)$ to $L^1(\Omega; \mathcal{H})$. Denote by ${\rm Dom}\, \delta$
the domain of $\delta$.
For an $\mathcal{H}$-valued random variable $v:\Omega\to \mathcal{H}$ and $F\in \mathbb{D}^{1,1}$, we write
\begin{align*}
D_vF=\langle DF\,, v\rangle_{\mathcal{H}}.
\end{align*}
The following criterion on the boundedness of the density of random variables was proved in \cite[Proposition 1]{CFN98};
see also \cite[Proposition 3.1]{KuN22}.

\begin{proposition}\label{prop:bound}
          Let $F\in \mathbb{D}^{1,1}$ and $v\in L^1(\Omega; \mathcal{H})$ be such that $D_vF \neq 0$ a.s. Assume that $v/D_vF\in {\rm Dom}\, \delta$. Then the 
          law of $F$ has a continuous and bounded density given by
          \begin{align*}
          f_F(a)=\E\left[\bm{1}_{\{F>a\}}\delta\left(\frac{v}{D_vF}\right)\right], \quad a\in \R
          \end{align*} 
\end{proposition}

\begin{remark}\label{rem:con}
          We have the following sufficient conditions for $v/D_vF\in {\rm Dom}\, \delta$ in Proposition \ref{prop:bound}:
          \begin{itemize}
          \item [(i)]$F\in\mathbb{D}^{2,p}$ and $(D_vF)^{-1}\in L^q(\Omega)$,
          \item  [(ii)] $v\in \mathbb{D}^{1,p}(\mathcal{H})$,
          \end{itemize}
          for some $p\geq 6$ and $q\geq 4$; see the discussions below Proposition 3.1 of \cite{KuN22}.
\end{remark}

According to Chen et al \cite[Proposition 5.1]{CHN21}, $u(t\,, x)\in\bigcap_{k\ge 2}\mathbb{D}^{1, k}$
	for all   $t>0$ and $x\in\R$. Moreover, by \cite[Lemma 2.1]{CKNP22}, for every $T >0$ and $k \geq 2$, there exists a real number $C_{T, k} >0$
	such that for $t\in(0\,,T]$ and $x \in \R$, and for every $(s\,, y)
	\in (0\,, t) \times \R$,
	\begin{align*}
		\|D_{s, y}u(t\,,x)\|_k \leq C_{T, k}\,\bm{p}_{t - s}(x - y)\bm{p}_s(y),
	\end{align*}
	which implies that 
	\begin{align}\label{derivative}
		\|D_{s, y}U(t\,,x)\|_k \leq C_{T, k}\,\bm{p}_{s(t-s)/t}\left(y- \frac{s}{t}x \right)
	\end{align}
	by \eqref{U} and \eqref{PPPP}.

\subsection{Association}
We recall the association property of solution to parabolic Anderson model. We refer to \cite{PR12} for more details on 
associated random variables. First, let us recall the definition of association; see Esary et al \cite{EPW67}.
A random vector $X:=(X_1\,,\ldots,X_m)$ is said to be \emph{associated} if
	\begin{equation}\label{E:assoc}
		\Cov[h_1(X)\,,h_2(X)]\ge0,
	\end{equation}
for every pair of functions $h_1,h_2:\R^m\to\R$
that are nondecreasing in every coordinate and satisfy $h_1(X),h_2(X)\in L^2(\Omega)$. A random field $\Phi=\{\Phi(x)\}_{x\in\R^d}$ is
	\emph{associated} if $(\Phi(x_1)\,,\ldots,\Phi(x_m))$ is associated
	for every $x_1,\ldots,x_m\in\R^d$.
	
Chen et al \cite{CKNP23} established the association property of solution to stochastic heat equation; see \cite[Theorem A.4]{CKNP23}. In particular, it implies that the solution $\{u(t\,,x): (t,x)\in (0, \infty)\times \R\}$ to \eqref{PAM} is associated. From the definition of association, we deduce the following. 
\begin{lemma}\label{lem:ass}
          The process $\{\log U(t\,,x): (t,x)\in (0, \infty)\times \R\}$ is associated. 
\end{lemma}

\section{Boundedness of density}\label{sec:bound}

The solution $u(t\,,x)$ to equation \eqref{PAM} is strictly positive (see \cite{Mue97} and \cite{MoF14}) and has a smooth probability density function on $(0, \infty)$ (see \cite{CHN21}). In this section, 
we show that this probability density function is bounded. 
\begin{proposition}\label{prop:density}
          For fixed $t>0$, the random variable $U(t\,,0)$ has a  continuous and bounded probability density function given by
          \begin{align*}
          f_{U(t\,,0)}(a)= \E\left[\bm{1}_{\{U(t\,,0)>a\}}\delta\left(\frac{v}{D_vU(t\,,0)}\right)\right], \quad a\in (0, \infty)
          \end{align*}
          where $v$ is chosen as
           \begin{align*}
           v(r\,,z) = \bm{1}_{\{0<r<t\}} \bm{p}_{r(t-r)/t}\left(z\right)U(r\,,z), \quad \text{for $(r\,,z)\in (0, \infty)\times \R$}.
           \end{align*}
\end{proposition}
\begin{proof}
            We will apply Proposition \ref{prop:bound} to prove this result. 
                       It is clear that the random variable $U(t\,,0)$ belongs to $\mathbb{D}^{1,1}$.
                                 From the choice of $v$, we have $v\in L^1(\Omega; \mathcal{H})$. Moreover, since the Malliavin derivative $D_{r,z}u(t\,,x)$ is nonnegative almost surely (see \cite[(A.4)]{CKNP23}) and
           $v(r\,,z)$ is strictly positive, 
           we see that
           \begin{align*}
           D_vU(t\,,0)= \int_0^t \int_\R v(r\,, z)D_{r,z} U(t\,,0)\d z \d r
           \end{align*}
           is strictly
           positive almost surely. According to Proposition \ref{prop:bound}, we need show that $v/D_vF\in {\rm Dom}\, \delta$, which reduces to show that the conditions (i) and (ii)
           in Remark \ref{rem:con} are satisfied. In order to check condition (ii) of Remark \ref{rem:con}, we use Minkowski's inequality to see that for all $p\geq 6$,
           \begin{align*}
           \left(\E[\|Dv\|^p_{\mathcal{H}^{\otimes2}}]\right)^{\frac2p}& = \left\| \int_0^t\d r\int_\R\d z\int_0^r\d r_1\int_\R\d z_1\, \bm{p}^2_{r(t-r)/t}\left(z \right)
           (D_{r_1,z_1}U(r\,,z))^2  \right\|_{\frac{p}2}\\
           & \leq  \int_0^t\d r\int_\R\d z\int_0^r\d r_1\int_\R\d z_1\, \bm{p}^2_{r(t-r)/t}\left(z \right)
           \left\|(D_{r_1,z_1}U(r\,,z))^2  \right\|_{\frac{p}2}\\
           & =  \int_0^t\d r\int_\R\d z\int_0^r\d r_1\int_\R\d z_1\, \bm{p}^2_{r(t-r)/t}\left(z \right)
           \left\|D_{r_1,z_1}U(r\,,z)  \right\|_{p}^2\\
           & \lesssim   \int_0^t\d r\int_\R\d z\int_0^r\d r_1\int_\R\d z_1\, \bm{p}^2_{r(t-r)/t}\left(z \right)
           \bm{p}^2_{r_1(r-r_1)/r}\left(z_1- \frac{r_1}{r}z \right),
           \end{align*}
           where the last inequality holds by \eqref{derivative}. The semigroup property of the heat kernel yields that
           \begin{align*}
           \left(\E[\|Dv\|^p_{\mathcal{H}^{\otimes2}}]\right)^{\frac2p}&\lesssim \int_0^t\d r \, \bm{p}_{2r(t-r)/t}\left(0\right)\int_0^r\d r_1\, \bm{p}_{2r_1(r-r_1)/r_1}\left(0\right)\\
           &= \int_0^t\d r \, \bm{p}_{2r(t-r)/t}\left(0\right)r\int_0^1\d r_1\, \bm{p}_{2r(1-r_1)}\left(0\right)\\
           &\asymp \int_0^t\sqrt{r}\, \bm{p}_{2r(t-r)/t}\left(0\right)\d r <\infty. 
           \end{align*}
           Hence, we have $v\in \mathbb{D}^{1,p}(\mathcal{H})$ for all $p\geq 6$. It is clear that $U(t\,,0) \in \mathbb{D}^{2,p}$ for all $p\geq 2$. Now it remains to verify that 
           \begin{align}\label{negative}
           (D_vU(t\,,0))^{-1} \in L^q(\Omega), \quad \text{for some $q\geq 4$}. 
           \end{align}
           
           We will adopt the method in \cite[Proposition 5.1]{KuN22} to prove \eqref{negative}. 
           Choose $\alpha \in (\frac43, 2)$. 
           Then we have 
           \begin{align}\label{eq:int}
           t-\epsilon^\alpha \geq \frac{t}2, \quad \text{for all $\epsilon \in (0, 1\wedge \frac{t}2)$}.
           \end{align}
           For $\epsilon \in (0, 1\wedge \frac{t}2)$, since the Malliavin derivative  $D_{r, z}U(t\,,0)$ is nonnegative almost surely. 
           We have 
           \begin{align*}
           D_vU(t\,,0) \geq \int_{t-\epsilon^\alpha}^t\int_\R v(r\,,z)D_{r, z}U(t\,,0)\d z \d r. 
           \end{align*}
           Because the Malliavin derivative of $U(t\,,0)$ satisfies 
           \begin{align*}
           D_{r, z}U(t\,,0) = \bm{p}_{r(t-r)/t}\left(z \right)U(r\,,z) + \int_r^t\int_\R
	 \bm{p}_{s(t-s)/t}\left(y\right)D_{r,z}U(s\,,y)\,
	\xi(\d s\, \d y)
           \end{align*}
           for $0<r<t$ and $D_{r, z}U(t\,,0)=0$ for $r\geq t$, we write
           \begin{align*}
           D_vU(t\,,0) \geq I_1+I_2,
           \end{align*}
           where
           \begin{align*}
           I_1&= \int_{t-\epsilon^\alpha}^t\int_\R \bm{p}^2_{r(t-r)/t}\left(z \right)U^2(r\,,z)\d z \d r,\\
           I_2&=
           \int_{t-\epsilon^\alpha}^t\int_\R\bigg( \bm{p}_{r(t-r)/t}\left(z \right)U(r\,,z) \int_r^t\int_\R
	 \bm{p}_{s(t-s)/t}\left(y\right)D_{r,z}U(s\,,y)\,
	\xi(\d s\, \d y) \bigg)\d z \d r. 
           \end{align*}
           Hence, we have for $\epsilon \in (0, 1\wedge \frac{t}2)$
           \begin{align}\label{I1I2}
           \P\{D_vU(t\,,0)< \epsilon\}\leq \P\{I_1<2\epsilon\} + \P\{I_2|>\epsilon\}.
           \end{align} 
           
           We first estimate $\P\{I_1<2\epsilon\}$. By Chebyshev's inequality,  for any $q\geq 2$, 
           \begin{align*}
           \P\{I_1<2\epsilon\}= \P\{I_1^{-1}>(2\epsilon)^{-1}\}\leq (2\epsilon)^q\E\left[\left( \int_{t-\epsilon^\alpha}^t\int_\R \bm{p}^2_{r(t-r)/t}\left(z \right)U^2(r\,,z)\d z \d r\right)^{-q}\right].
           \end{align*}
           Set 
           \begin{align}\label{eq:m}
           m(\epsilon)&:= \int_{t-\epsilon^\alpha}^t\int_\R \bm{p}^2_{r(t-r)/t}\left(z\right)\d z\d r\nonumber\\
           &= \int_{t-\epsilon^\alpha}^t \bm{p}_{2r(t-r)/t}\left(0\right)\d r \asymp  \int_{t-\epsilon^\alpha}^t \frac{\d r}{\sqrt{t-r}} \qquad \text{because of \eqref{eq:int}}\nonumber\\
           &\asymp \epsilon^{\alpha/2}.
           \end{align}
           Since the function $x\to x^{-q}$ is convex, we apply Jensen's inequality with respect to the probability measure 
           $\frac{1}{m(\epsilon)} \bm{p}^2_{r(t-r)/t}\left(z\right)\d z\d r$ to write
           \begin{align*}
           &\E\left[\left( \frac{1}{m(\epsilon)}\int_{t-\epsilon^\alpha}^t\int_\R \bm{p}^2_{r(t-r)/t}\left(z\right)U^2(r\,,z)\d z \d r\right)^{-q}\right]\\
           &\qquad \qquad\qquad \leq \frac{1}{m(\epsilon)}\int_{t-\epsilon^\alpha}^t\int_\R \bm{p}^2_{r(t-r)/t}\left(z \right)\E[U^{-2q}(r\,,z)]\d z \d r\\
           &\qquad \qquad \qquad \lesssim  \frac{1}{m(\epsilon)}\int_{t-\epsilon^\alpha}^t\int_\R \bm{p}^2_{r(t-r)/t}\left(z \right)\d z \d r=1,
                      \end{align*}
           where the second inequality is due to \cite[(72)]{KuN22}; see also \cite[Theorem 1.4]{ChK17}. Hence, we obtain that
           for any $q\geq 2$,
           \begin{align}\label{eq:I1}
           \P\{I_1<2\epsilon\} \lesssim \epsilon^{q}m(\epsilon)^{-q} \asymp \epsilon^{(1-\alpha/2)q},
           \end{align}
           thanks to \eqref{eq:m}.

           Next, we estimate $\P\{I_2|>\epsilon\}$. Using stochastic Fubini's theorem, we write
           \begin{align*}
           I_2= \int_{t-\epsilon^\alpha}^t\int_\R\left[ \bm{p}_{s(t-s)/t}\left(y \right) 
           \int^s_{t-\epsilon^\alpha}\int_\R  \bm{p}_{r(t-r)/t}\left(z \right)U(r\,,z) D_{r,z}U(s\,,y)\d z\d r
           \right]\xi(\d s\,\d y).
           \end{align*}
           By Burkholder's inequality and Minkowski's inequality, for any $q\geq 2$,
           \begin{align*}
           \|I_2\|_q^2 &\lesssim \left\|\int_{t-\epsilon^\alpha}^t\int_\R\bm{p}^2_{s(t-s)/t}\left(y\right) 
           \left[ \int^s_{t-\epsilon^\alpha}\int_\R  \bm{p}_{r(t-r)/t}\left(z \right)U(r\,,z) D_{r,z}U(s\,,y)\d z\d r
           \right]^2\d y\d s\right\|_{\frac{q}{2}}\\
           &\leq \int_{t-\epsilon^\alpha}^t\int_\R\bm{p}^2_{s(t-s)/t}\left(y \right) 
           \left\| \int^s_{t-\epsilon^\alpha}\int_\R  \bm{p}_{r(t-r)/t}\left(z \right)U(r\,,z) D_{r,z}U(s\,,y)\d z\d r
           \right\|_q^2\d y\d s\\
           &\leq  \int_{t-\epsilon^\alpha}^t\int_\R\bm{p}^2_{s(t-s)/t}\left(y \right) 
           \left[\int^s_{t-\epsilon^\alpha}\int_\R  \bm{p}_{r(t-r)/t}\left(z\right)\|U(r\,,z) D_{r,z}U(s\,,y)\|_q\d z\d r
           \right]^2\d y\d s.
           \end{align*}
           Using H\"{o}lder's inequality, \eqref{moment} and \eqref{derivative},  we see that
           \begin{align*}
           \|U(r\,,z) D_{r,z}U(s\,,y)\|_q \leq            \|U(r\,,z)\|_{2q} \|D_{r,z}U(s\,,y)\|_{2q} \lesssim \bm{p}_{r(s-r)/s}\left(z-\frac{r}{s}y\right).
           \end{align*}
           Thus, the preceding implies that
           \begin{align}\label{I21}
           \|I_2\|_q^2 &\lesssim
             \int_{t-\epsilon^\alpha}^t\int_\R\bm{p}^2_{s(t-s)/t}\left(y \right) 
           \left[\int^s_{t-\epsilon^\alpha}\int_\R  \bm{p}_{r(t-r)/t}\left(z\right)\bm{p}_{r(s-r)/s}\left(z-\frac{r}{s}y\right)\d z\d r
           \right]^2\d y\d s\nonumber\\
           &= \int_{t-\epsilon^\alpha}^t\int_\R\bm{p}^2_{s(t-s)/t}\left(y \right) 
           \left[\int^s_{t-\epsilon^\alpha}  \bm{p}_{[r(t-r)/t] + [r(s-r)/s]}\left(\frac{r}{s}y \right)\d r
           \right]^2\d y\d s \nonumber\\
           &= \int_{t-\epsilon^\alpha}^t\int_\R\bm{p}^2_{s(t-s)/t}\left(y \right) 
           \left[\int^s_{t-\epsilon^\alpha}  \frac{s}{r} \, \bm{p}_{(s^2[r(t-r)/t] + s^2[r(s-r)/s])/r^2}\left(y \right)\d r
           \right]^2\d y\d  s
                      \end{align}
           where the first equality holds by the semigroup property of heat kernel and the second equality follows from the 
           following identity
           \begin{align}\label{eq:scale}
           \bm{p}_t(\alpha x)= \alpha^{-1}\bm{p}_{t/\alpha^2}(x) \qquad[t, \alpha>0, x\in \R].
           \end{align}
           Denote
           \begin{align*}
           \lambda(r)= (s^2[r(t-r)/t] + s^2[r(s-r)/s])/r^2.
           \end{align*}
           We use  the formula
	\begin{equation} \label{E:equa2}
		\bm{p} _\sigma(x)     \bm{p} _\tau (x)   = 
		\bm{p}_{\sigma+\tau}(0) \bm{p}_{\sigma \tau/(\sigma+\tau)}(x)
		\qquad [\sigma, \tau >0, x\in \R],
	\end{equation}
	to derive that
	           \begin{align}\label{I22}
                      & \int_{t-\epsilon^\alpha}^t\int_\R\bm{p}^2_{s(t-s)/t}\left(y\right) 
           \left[\int^s_{t-\epsilon^\alpha}  \frac{s}{r} \, \bm{p}_{\lambda(r)}\left(y\right)\d r
           \right]^2\d y\d s\nonumber\\
                      &\quad= \int_{t-\epsilon^\alpha}^t\d s\int_\R\d y\, \bm{p}_{2s(t-s)/t}\left(0\right)  \bm{p}_{s(t-s)/(2t)}\left(y\right) 
           \int_{[t-\epsilon^\alpha, s]^2}  \frac{s\d r_1}{r_1} \frac{s\d r_2}{r_2} \, \bm{p}_{\lambda(r_1)+\lambda(r_2)}\left(0\right)
           \nonumber\\
           &\qquad\qquad\qquad\qquad \qquad \qquad \qquad 
           \bm{p}_{\lambda(r_1)\lambda(r_2)/(\lambda(r_1)+\lambda(r_2))}\left(y\right)\nonumber\\
                                 &\quad= \int_{t-\epsilon^\alpha}^t\d s\, \bm{p}_{2s(t-s)/t}\left(0\right)
           \int_{[t-\epsilon^\alpha, s]^2}  \frac{s\d r_1}{r_1} \frac{s\d r_2}{r_2} \bm{p}_{\lambda(r_1)+\lambda(r_2)}\left(0\right)
             \bm{p}_{[s(t-s)/(2t)] + [\lambda(r_1)\lambda(r_2)/(\lambda(r_1)+\lambda(r_2))]}\left(0\right),
                       \end{align}
           where we used semigroup property of heat kernel in the second equality.  
           Notice that 
           \begin{align*}
           \bm{p}_{[s(t-s)/(2t)] + [\lambda(r_1)\lambda(r_2)/(\lambda(r_1)+\lambda(r_2))]}\left(0\right)
           \geq \bm{p}_{\lambda(r_1)\lambda(r_2)/(\lambda(r_1)+\lambda(r_2))}\left(0\right).
           \end{align*}
           Thus, in view of \eqref{I21} and \eqref{I22}, 
           it follows that
           	           \begin{align*}
           \|I_2\|_q^2 
           &\lesssim   \int_{t-\epsilon^\alpha}^t\d s\, \bm{p}_{2s(t-s)/t}\left(0\right)   
           \int_{[t-\epsilon^\alpha, s]^2}  \frac{s\d r_1}{r_1} \frac{s\d r_2}{r_2} \bm{p}_{\lambda(r_1)+\lambda(r_2)}\left(0\right)
           \bm{p}_{\lambda(r_1)\lambda(r_2)/(\lambda(r_1)+\lambda(r_2))}\left(0\right)\\
           &\asymp  \int_{t-\epsilon^\alpha}^t\d s\, \bm{p}_{2s(t-s)/t}\left(0\right)   
           \left[\int_{t-\epsilon^\alpha}^s \bm{p}_{\lambda(r)}(0)  \frac{s\d r}{r}\right]^2.
           \end{align*}

            Now for all $s\in [t-\epsilon^\alpha, t]$ and $r\in [t-\epsilon^\alpha, s]$, 
           we have $s\in [\frac{t}{2}, t]$ and $r\in [\frac{t}{2}, t]$; see \eqref{eq:int}. Moreover, for $s, r\in [\frac{t}{2}, t]$, we have 
           \begin{align*}
           \lambda(r)\geq s^2r(t-r)/(tr^2) =\frac{s(t-r)}{tr} \gtrsim t-r. 
           \end{align*}
           Therefore, we conclude that 
           	           \begin{align*}
           \|I_2\|_q^2 
           &\lesssim \int_{t-\epsilon^\alpha}^t \frac{\d s}{\sqrt{t-s}} \left[\int_{t-\epsilon^\alpha}^t \frac{\d r}{\sqrt{t-r}}\right]^2 \asymp \epsilon^{\frac{3}{2}\alpha}.
           \end{align*}
         Hence, by Chebyshev's inequality, for any $q\geq 2$,
           \begin{align} \label{eq:I2}
           \P\{I_2|>\epsilon\} &\leq \epsilon^{-q}\E[|I_2|^q]\lesssim \epsilon^{(\frac{3}{4}\alpha-1)q}.
           \end{align}

           Letting $\alpha = \frac85$, we deduce from \eqref{eq:I1}, \eqref{eq:I2} and \eqref{I1I2} that for any $q\geq 2$,
           \begin{align}\label{smallprob}
           \P\{D_vU(t\,,0)< \epsilon\}\lesssim \epsilon^{\frac{q}{5}},
           \end{align}
           for $\epsilon \in (0, 1\wedge \frac{t}2)$, where the implicit constant only depends on $t$ and $q$.  The estimate \eqref{smallprob} ensures that $D_vU(t\,,0)$
           has finite negative moments of all orders. This verifies \eqref{negative} and completes the proof. 
\end{proof}

\begin{remark}
          We refer to \cite{GaH22} and \cite{HuL22} for the asymptotic behavior of the density of solution to parabolic Anderson model.
\end{remark}

\begin{corollary}\label{cor:density2}
          For fixed $t>0$ and $x\in \R$, the random variable $\log U(t\,,x)$ has a bounded probability density function on $\R$. 
\end{corollary}
\begin{proof}
          By stationary, we assume $x=0$. Denote $X=U(t\,,0)$ and $Y=\log U(t\,,0)$. Then  we have the following relation between the probability density functions of $X$ and $Y$:
          \begin{align}\label{d1}
          f_Y(y)= f_X(\e^y)\e^y, \quad \text{for all $y\in\R$}. 
          \end{align}
          In light of Proposition \ref{prop:density}, 
          \begin{align}\label{d2}
          f_X(\e^y)=\E\left[\bm{1}_{\{U(t\,,0)>\e^y\}}\delta\left(\frac{v}{D_vU(t\,,0)}\right)\right] \leq \left\| \delta\left(\frac{v}{D_vU(t\,,0)}\right) \right\|_2 
          \left(\P\{U(t\,,0)>\e^y\}\right)^{1/2}
          \end{align}
          thanks to the Cauchy-Schwarz's inequality. According to \cite[Theorem 1.11]{CoG20}, there exist positive constants $c_1, c_2$ and $y_0$ such that 
          \begin{align}\label{d3}
          \P\{U(t\,,0)>\e^y\} \leq c_1\e^{-c_2y^{3/2}}, \quad \text{for all $y\geq y_0$}.
          \end{align}
          Thus, we see from \eqref{d1}, \eqref{d2} and \eqref{d3} that 
          \begin{align*}
          \sup_{y\geq y_0}f_Y(y) <\infty. 
          \end{align*}
          Moreover, since $f_X$ is bounded by Proposition \ref{prop:density}, we see from \eqref{d1} tha $\sup_{y\leq y_0}f_Y(y)<\infty$. 
          Therefore, we conclude that the probability density function $f_Y$ is bounded. 
\end{proof}

\begin{corollary}\label{cor:density}
         Fix $t>0$. There exists a positive constant $K$ such that for all $x\in \R$
         \begin{align}\label{eq:prob}
         &\sup_{a, b\in \R}\left(\P\left\{\log U(t\,,x) \leq a, \log U(t\,,0) \leq b\right\} - \P\left\{\log U(t\,,x) \leq a\right\}\P\left\{\log U(t\,,0) \leq b\right\}\right)\nonumber \\
         &\qquad\qquad\qquad \qquad\qquad\qquad\qquad\qquad\qquad\qquad\qquad
         \leq K\left[\Cov(\log U(t\,,x)\,, \log U(t\,,0))\right]^{1/3}.
         \end{align}
\end{corollary}
\begin{proof}
          Because the process $\{\log U(t\,,x): (t,x)\in (0, \infty)\times \R\}$ is associated (see Lemma \ref{lem:ass}) and the random variable 
           $\log U(t\,,x)$ has a bounded probability density function (see Corollary \ref{cor:density2}),
          the above inequality follows immediately from  \cite[(6.2.20)]{PR12} (see also \cite[Theorem 6.2.15]{PR12}).
\end{proof}

\section{Proof of Theorem \ref{th:main}}\label{sec:main}

We prove Theorem \ref{th:main} in this section. First, let us recall a sharp upper tail estimates for the solution to parabolic Anderson model established in \cite{GaH22}. Let $\{\mathcal{Z}(t\,,x):(t,x)\in (0, \infty)\times \R\}$  solves 
\begin{align}\label{PAM2}
\begin{cases}
\partial_t\mathcal{Z}(t\,,x)= \frac14\partial_x^2\mathcal{Z}(t\,,x) + \mathcal{Z}(t\,,x)\, \xi(t\,,x),\\
\mathcal{Z}(0)=\delta_0.
\end{cases}
\end{align}
According to \cite[Theorem 2]{GaH22}, for fixed $t>0$, there exist $\theta_0>0$ and $C<\infty$ such that for all $\theta>\theta_0$,
\begin{align*}
\exp\left(-\frac43\theta^{3/2}-\theta^{1/2}\log\theta\right)\leq\P\left( \frac{\log \mathcal{Z}(t\,,0) + \frac{t}{12}}{t^{1/3}}\geq \theta \right) \leq \exp\left(-\frac43\theta^{3/2} + C\theta^{3/4}\right).
\end{align*}
Notice that the result of \cite[Theorem 2]{GaH22} is for $t\in [t_0, \infty)$ with $t_0>0$. But $t_0$ maybe chosen to be any positive value; see the discussions at the beginning of Section 1.5 of \cite{GaH22}. Hence, we can have the above estimate for fixed $t>0$. The above estimate implies that for fixed $t>0$, 
\begin{align*}
\lim_{\theta\to+\infty} \frac{\log \P\left( \log \mathcal{Z}(t\,,0) +\frac{t}{12} \geq t^{1/3}\theta \right)}{\theta^{3/2}}= -\frac43
\end{align*}
and hence
\begin{align}\label{tail1}
\lim_{\theta\to+\infty} \frac{\log \P\left( \log \mathcal{Z}(t\,,0) \geq \theta \right)}{\theta^{3/2}}= -\frac{4}{3\sqrt{t}}.
\end{align}
The solution to \eqref{PAM2} satisfies 
\begin{align*}
	\mathcal{Z}(t\, , x) = \bm{p}_{t/2}(x) + \int_0^t\int_\R\bm{p}_{(t - s)/2}(x - y)\mathcal{Z}(s\,, y)\,\xi(\d s\,\d y).
\end{align*}
Using \eqref{eq:scale} and the fact that $\{\frac12\xi(s/2\,,y/2): (s,y)\in (0, \infty)\times \R)\}$ has the distribution as space-time white noise, it is readily to see that 
\begin{align*}
	\frac12\mathcal{Z}(t/2\, , x/2) = \bm{p}_{t}(x) + \int_0^t\int_\R\bm{p}_{t - s}(x - y)\frac12
	\mathcal{Z}(s/2\,, y/2)\,\xi(\d s\,\d y).
\end{align*}
Therefore, we conclude that $\frac12\mathcal{Z}(t/2\, , x/2)$ has the distribution as $u(t\,,x)$ which solves \eqref{PAM}. By \eqref{tail1}, we obtain that
\begin{align}\label{eq:tail1}
\lim_{\theta\to+\infty} \frac{\log \P\left( \log U(t\,,0) \geq \theta \right)}{\theta^{3/2}}= -\frac{4}{3}\sqrt{\frac2t}.
\end{align}

We next  give some estimate on the covariance $\Cov(\log U(t\,,x)\,, \log U(t\,,0))$. 
\begin{lemma}\label{lem:cov:decay}
          For fixed $t>0$, there exists a positive constant $c$ depending on $t$ such that for all $x\in \R$
          \begin{align}\label{eq:cov}
          \left|\Cov(\log U(t\,,x)\,, \log U(t\,,0))\right| \leq \frac{c}{|x|}.
          \end{align}
\end{lemma}
\begin{proof}
          Since $U(t\,,x)$ belongs to $\mathbb{D}^{1,2}$ and has finite negative moments of all orders (see \cite{CHN21}), 
          the random variable $\log U(t\,,x)$ belongs to $\mathbb{D}^{1,2}$ and we have 
          \begin{align*}
          D_{r,z}\log U(t\,,x) =\bm{1}_{\{0<r<t\}} \frac{D_{r,z}U(t\,,x)}{U(t\,,x)}
          \end{align*}
           (see \cite[Lemma 4.3(i)]{CKNP23}). Applying the Poincar\'{e}
          inequality \eqref{Poincare:Cov}, 
          \begin{align*}
          \left|\Cov(\log U(t\,,x)\,, \log U(t\,,0))\right| &\leq \int_0^{t}\int_\R \left\|\frac{D_{r,z}U(t\,,x)}{U(t\,,x)}\right\|_2
          \left\|\frac{D_{r,z}U(t\,,0)}{U(t\,,0)}\right\|_2\d z\d r\\
          &\leq \left\|U(t\,,0)^{-1}\right\|^2_4 \int_0^{t}\int_\R \left\|D_{r,z}U(t\,,x)\right\|_4
          \left\|D_{r,z}U(t\,,0)\right\|_4\d z\d r\\
          &\lesssim \int_0^t\int_\R \bm{p}_{r(t-r)/t}\left(z-\frac{r}{t}x\right) \bm{p}_{r(t-r)/t}\left(z\right)\d z\d r,
          \end{align*}
          where the second inequality follows from H\"older's inequality and the third inequality is due to \eqref{derivative}.
          By the semigroup property of heat kernel, it follows that
          \begin{align}\label{eq:cov0}
          \left|\Cov(\log U(t\,,x)\,, \log U(t\,,0))\right| &\lesssim \int_0^t\bm{p}_{2r(t-r)/t}\left(\frac{r}{t}x\right)\d r 
          = t\int_0^1\bm{p}_{2tr(1-r)}\left(rx\right)\d r\nonumber\\
          &\asymp \int_0^1\e^{-\frac{rx^2}{4t(1-r)}}\frac{\d r}{\sqrt{r(1-r)}}.
          \end{align}
          Using change of variable, 
          \begin{align}\label{eq:cov1}
          \int_0^{1/2}\e^{-\frac{rx^2}{4t(1-r)}}\frac{\d r}{\sqrt{r(1-r)}} & \leq \int_0^{1/2}\e^{-\frac{rx^2}{4t}}\frac{\d r}{\sqrt{r/2}}\nonumber\\
          &\asymp \frac{1}{|x|}\int_0^{x^2/2} \e^{-\frac{r}{4t}}\frac{\d r}{\sqrt{r}} \leq  \frac{1}{|x|}\int_0^{\infty} \e^{-\frac{r}{4t}}\frac{\d r}{\sqrt{r}}  \asymp \frac{1}{|x|}. 
          \end{align}
          Similarly, 
          \begin{align}\label{eq:cov2}
          \int_{1/2}^1\e^{-\frac{rx^2}{4t(1-r)}}\frac{\d r}{\sqrt{r(1-r)}} & \leq \int_{1/2}^1\e^{-\frac{x^2}{8t(1-r)}}\frac{\d r}{\sqrt{(1-r)/2}}
          =\int_0^{1/2}\e^{-\frac{x^2}{8tr}}\frac{\d r}{\sqrt{r/2}}\nonumber\\
          &= |x|\int_0^{\frac{1}{2x^2}} \e^{-\frac{1}{8tr}}\frac{\d r}{\sqrt{r/2}} \leq |x|\e^{-\frac{x^2}{4t}}\int_0^{\frac{1}{2x^2}}\frac{\d r}{\sqrt{r/2}} \asymp \e^{-\frac{x^2}{4t}}. 
          \end{align}
          We obtain \eqref{eq:cov} from \eqref{eq:cov1}, \eqref{eq:cov2} and \eqref{eq:cov0}.
\end{proof}

We are now ready to prove Theorem \ref{th:main}. 
\begin{proof}[{{Proof of Theorem \ref{th:main}}}]
          Fix $t>0$. Let $\beta$ be a positive number that is strictly less than $\frac18\sqrt{\frac{t}{2}}$. Choose and fix
          $a\in (0, \frac16)$, $\epsilon \in (0, 1)$ such that
          \begin{align}\label{exponent1}
          \beta<\frac{a}{\frac43\sqrt{\frac2{t}} + \epsilon}.
          \end{align}
          For this fixed $\epsilon \in (0, 1)$, by \eqref{eq:tail1}, there exists a positive constant $\tilde{C}$ 
          depending on $\epsilon$ and the fixed $t>0$ such that
          \begin{align}\label{eq:tail2}
           \P\left( \log U(t\,,0) \geq \theta \right) \geq  \e^{-(\frac{4}{3}\sqrt{\frac2t} +\epsilon)\theta^{3/2}}, 
           \quad \text{for all $\theta\geq \tilde{C}$}.
          \end{align}

          Define $x_j= 2jR/\floor*{R^a}$ 
          for $j=1, \ldots, \floor*{R^a}$. Here the notation $\floor*{x}$ denotes the largest integer which is 
          less than or equal to $x$. 
          Assume that $R$ is sufficiently large so that $\floor*{R^a} \geq R^a/2$ and $(\beta\log R)^{2/3} \geq \tilde{C}$, where
          $\tilde{C}$ is the constant in \eqref{eq:tail2}.

          By stationarity of $\{U(t\,,x): x\in \R\}$,
          \begin{align}\label{eq:prob1}
          &\P\left\{\max_{0\leq x\leq 2R}\log U(t\,,x) \leq (\beta\log R)^{2/3}\right\}\leq 
          \P\left\{\max_{1\leq j\leq \floor*{R^a}}\log U(t\,,x_j) \leq (\beta\log R)^{2/3}\right\}\nonumber\\
          &\qquad=\P\left\{\max_{1\leq j\leq \floor*{R^a}}\log U(t\,,x_j) \leq (\beta\log R)^{2/3}\right\} -
          \prod_{j=1}^{\floor*{R^a}}\P\left\{\log U(t\,,x_j) \leq (\beta\log R)^{2/3}\right\}\nonumber\\
          &\qquad\qquad\qquad\qquad\qquad + \left(1-\P\left\{\log U(t\,,0) > (\beta\log R)^{2/3}\right\}\right)^{\floor*{R^a}}
          \end{align}
          Since the process $\{\log U(t\,,x): x\in \R\}$ is associated (see Lemma \ref{lem:ass}), by \cite[(2.6)]{Pu23},
          \begin{align*}
          &\P\left\{\max_{1\leq j\leq \floor*{R^a}}\log U(t\,,x_j) \leq (\beta\log R)^{2/3}\right\} -
          \prod_{j=1}^{\floor*{R^a}}\P\left\{\log U(t\,,x_j) \leq (\beta\log R)^{2/3}\right\}\\
          &\qquad \leq \sum_{1\leq j<k\leq \floor*{R^a}}\bigg(
          \P\left\{\log U(t\,,x_j) \leq (\beta\log R)^{2/3}, \log U(t\,,x_k) \leq (\beta\log R)^{2/3}\right\}\\
          &\qquad\qquad\qquad\qquad \qquad - \P\left\{\log U(t\,,x_j) \leq (\beta\log R)^{2/3}\right\}
          \P\left\{\log U(t\,,x_k) \leq (\beta\log 2R)^{2/3}\right\}
          \bigg)\\
          &\qquad \lesssim  \sum_{1\leq j<k\leq \floor*{R^a}} \left[\Cov(\log U(t\,,x_j)\,, \log U(t\,,x_k))\right]^{1/3}\\
          &\qquad \lesssim  \sum_{1\leq j<k\leq \floor*{R^a}} \frac{1}{|x_k-x_j|^{1/3}} \asymp R^{(a-1)/3}
           \sum_{1\leq j<k\leq \floor*{R^a}} \frac{1}{|k-j|^{1/3}},
          \end{align*}
          where the second inequality holds by Corollary \ref{cor:density} and stationarity, and the third inequality holds by 
          Lemma \ref{lem:cov:decay}. Evidently, we have
          \begin{align*}
          \sum_{1\leq j<k\leq \floor*{R^a}} \frac{1}{|k-j|^{1/3}} \lesssim R^{\frac{5a}{3}}.
          \end{align*}
          Hence, we obtain that
          \begin{align*}
          &\P\left\{\max_{1\leq j\leq \floor*{R^a}}\log U(t\,,x_j) \leq (\beta\log R)^{2/3}\right\} -
          \prod_{j=1}^{\floor*{R^a}}\P\left\{\log U(t\,,x_j) \leq (\beta\log R)^{2/3}\right\}\lesssim R^{\frac13-2a},
          \end{align*}
          which together with \eqref{eq:prob1} implies that
          \begin{align*}
          \P\left\{\max_{0\leq x\leq 2R}\log U(t\,,x) \leq (\beta\log R)^{2/3}\right\}&\lesssim R^{\frac13-2a}+
          \left(1-\P\left\{\log U(t\,,0) > (\beta\log R)^{2/3}\right\}\right)^{\floor*{R^a}}\\
          &\leq R^{\frac13-2a}+\left(1- \e^{-(\frac{4}{3}\sqrt{\frac2t} +\epsilon)\beta\log R} \right)^{\floor*{R^a}}\\
          &\leq R^{\frac13-2a}+ \e^{-\frac12 R^{a-\frac{4}{3}(\sqrt{\frac2t} +\epsilon)\beta}},
          \end{align*}
          where the second inequality is due to \eqref{eq:tail2} and the third inequality follows from the fact 
          that $1-x\leq \e^{-x}$ for all $x\geq0$.
          Since $\frac13-2a<0$ and $a-\frac{4}{3}(\sqrt{\frac2t} +\epsilon)\beta>0$ by \eqref{exponent1},
          letting $R=2^n$, we obtain that
          \begin{align*}
          \sum_{n=1}^\infty\P\left\{\max_{0\leq x\leq 2^{n+1}}\log U(t\,,x) \leq (\beta\log 2^{n})^{2/3}\right\} <\infty. 
          \end{align*}
          Hence, by Borel-Cantelli's lemma, we have almost surely,
          \begin{align*}
          \liminf_{n\to\infty} \frac{\max_{0\leq x\leq 2^{n+1}}\log U(t\,,x)}{(\log 2^n)^{2/3}} \geq \beta^{2/3}.
          \end{align*}
          Since the quantity $\max_{0\leq x\leq 2R}\log U(t\,,x)$ is monotone in $R$, we conclude that almost surely,
          \begin{align*}
          \liminf_{R\to\infty} \frac{\max_{0\leq x\leq 2R}\log U(t\,,x)}{(\log R)^{2/3}} \geq \beta^{2/3}.
          \end{align*}
          Letting $\beta\uparrow \frac18\sqrt{\frac{t}{2}}$, it follows that almost surely
          \begin{align*}
          \liminf_{R\to\infty} \frac{\max_{0\leq x\leq 2R}\log U(t\,,x)}{(\log R)^{2/3}} \geq \frac14\left(\frac{t}2\right)^{1/3},
          \end{align*}
          which is equivalent to \eqref{eq:lower} by stationarity.

         The proof is complete.
\end{proof}

\noindent\textbf{Acknowledgement}.  Research supported in part by National Natural Science Foundation of China (No. 12201047), Beijing Natural Science Foundation (No. 1232010) and 
National Key R\&D Program of China (No. 2022YFA 1006500).

 \bigskip

 \begin{small}
\noindent\textbf{Fei Pu}
Laboratory of Mathematics and Complex Systems,
School of Mathematical Sciences, Beijing Normal University, 100875, Beijing, China.\\
Email: \texttt{fei.pu@bnu.edu.cn}\\
\end{small}

\end{document}